\documentclass[10pt,a4paper,twoside]{amsart}
\usepackage{amsmath,amsfonts,amsthm,amsopn,color,amssymb,enumitem}
\usepackage{palatino}
\usepackage{graphicx}
\usepackage[colorlinks=true]{hyperref}
\hypersetup{urlcolor=blue, citecolor=red, linkcolor=blue}

\usepackage{cite}

\newcommand{\e}{\varepsilon}

\newcommand{\C}{\mathbb{C}}
\newcommand{\R}{\mathbb{R}}

\newcommand{\RD}{{\mathbb{R}^2}}

\newcommand{\weakto}{\rightharpoonup}

\renewcommand{\le}{\leslant}
\renewcommand{\ge}{\geslant}
\renewcommand{\a }{\alpha }

\renewcommand{\b }{\beta }

\renewcommand{\d }{\delta }

\newcommand{\g }{\gamma }

\renewcommand{\l }{\lambda}
\newcommand{\n }{\nabla }
\newcommand{\s }{\sigma }

\renewcommand{\H}{H^1(\RD)}
\newcommand{\Hr}{H^1_r(\RD)}

\newcommand{\I}{\mathcal{I}}

\newcommand{\N}{\mathbb{N}}

\renewcommand{\C}{\mathbb{C}}
\renewcommand{\o}{\omega}

\def\bbm[#1]{\mbox{\boldmath $#1$}}
\newcommand{\beq }{\begin{equation}}
\newcommand{\eeq }{\end{equation}}

\renewcommand{\le}{\leqslant}
\renewcommand{\ge}{\geqslant}
\newcommand{\dis}{\displaystyle}

\newcommand{\ir}{\int_{-\infty}^{+\infty}}
\newcommand{\ird}{\int_{\R^2}}

\newtheorem{theorem}{Theorem}[section]
\newtheorem{lemma}[theorem]{Lemma}

\newtheorem{proposition}[theorem]{Proposition}
\newtheorem{remark}[theorem]{Remark}

\title[A Variational Analysis of a Gauged Nonlinear Schr\"{o}dinger Equation]
{A Variational Analysis of a Gauged Nonlinear Schr\"{o}dinger Equation}

\author[Pomponio]{Alessio Pomponio$^1$}
\address{$^1$Dipartimento di Meccanica, Matematica e Management, Politecnico di
Bari, Via E. Orabona 4, 70125 Bari, Italy.}
\author[Ruiz]{David Ruiz$^2$}
\address{$^2$Dpto. An\'{a}lisis Matem\'{a}tico, Granada, 18071 Spain.}

\thanks{A.P. is supported by M.I.U.R. - P.R.I.N.
``Metodi variazionali e topologici nello studio di fenomeni non
lineari'', by GNAMPA Project ``Metodi Variazionali e Problemi
Ellittici Non Lineari'' and by FRA2011 ``Equazioni ellittiche di
tipo Born-Infeld". D.R. is supported by the Spanish Ministry of
Science and Innovation under Grant MTM2011-26717 and by J.
Andalucia (FQM 116).}

\email{a.pomponio@poliba.it, daruiz@ugr.es}
\date{}

\keywords{Gauged Schr\"{o}dinger Equations, Chern-Simons theory,
Variational methods, concentration compactness.}

\subjclass[2010]{35J20, 35Q55.}

\begin{document}

\begin{abstract}

This paper is motivated by a gauged Schr\"{o}dinger equation in
dimension 2 including the so-called Chern-Simons term. The study
of radial stationary states leads to the nonlocal
problem:
$$  - \Delta u(x) + \left( \o + \dis \frac{h^2(|x|)}{|x|^2} +  \int_{|x|}^{+\infty} \frac{h(s)}{s} u^2(s)\, ds   \right)  u(x)  = |u(x)|^{p-1}u(x), $$
where
$$ h(r)= \frac{1}{2}\int_0^{r} s u^2(s)  \, ds.$$

This problem is the Euler-Lagrange equation of a
certain energy functional. In this paper the study of the global
behavior of such functional is completed. We show that for
$p\in(1,3)$, the functional may be bounded from below or not, depending on $\o $. Quite surprisingly, the
threshold value for $\o $ is explicit. From this study we
prove existence and non-existence of positive solutions.

\end{abstract}

\maketitle

\section{Introduction}

In this paper we are concerned with a planar gauged Nonlinear
Schr\"{o}dinger Equation: \beq \label{planar} i D_0\phi +
(D_1D_1 +D_2D_2)\phi + |\phi|^{p-1} \phi =0. \eeq

Here $t \in \R$, $x=(x_1, x_2) \in \R^2$, $\phi :
\R \times \R^2 \to \C$ is the scalar field, $A_\mu : \R\times \R^2
\to \R$ are the components of the gauge potential and $D_\mu =
\partial_\mu + i A_\mu$ is the covariant derivative ($\mu = 0,\
1,\ 2$).

The classical equation for the gauge potential $A_\mu$ is the Maxwell
equation. However, the modified gauge field equation proposes to
include the so-called Chern-Simons term into that equation (see for instance \cite[Chapter 1]{tar}):

\beq \label{gauge field} \partial_\mu F^{\mu \nu} + \frac 1 2
\kappa \epsilon^{\nu \alpha \beta} F_{\alpha \beta}= j^\nu, \ \mbox{ with } \ F_{\mu \nu}=
\partial_\mu A_\nu - \partial_\nu A_\mu. \eeq

In the above equation, $\kappa$ is a parameter that
measures the strength of the Chern-Simons term. As usual,
$\epsilon^{\nu \alpha \beta}$ is the Levi-Civita tensor, and
super-indices are related to the Minkowski metric with signature
$(1,-1,-1)$. Finally, $j^\mu$ is the conserved matter current,

$$ j^0= |\phi|^2, \ j^i= 2 {\rm Im} \left (\bar{\phi} D_i \phi \right).$$

At low energies, the Maxwell term becomes negligible and can be dropped, giving rise to: \beq \label{cs}
\frac 1 2 \kappa \epsilon^{\nu \alpha \beta} F_{\alpha \beta}=
j^\nu. \eeq See \cite{hagen, hagen2, jackiw0, jackiw, jackiw2} for
the discussion above.

For the sake of simplicity, let us fix $\kappa = 2$. Equations
\eqref{planar} and \eqref{cs} lead us to the problem:
 \beq \label{eq:e0}
\begin{array}{l}
i D_0\phi + (D_1D_1 +D_2D_2)\phi + |\phi|^{p-1} \phi =0,\\
\partial_0 A_1  -  \partial_1 A_0  = {\rm Im}( \bar{\phi}D_2\phi), \\
\partial_0 A_2  - \partial_2 A_0 = -{\rm Im}( \bar{\phi}D_1\phi), \\
\partial_1 A_2  -  \partial_2 A_1 =  \frac 1 2 |\phi|^2. \end{array}
\eeq

As usual in Chern-Simons theory, problem \eqref{eq:e0} is
invariant under gauge transformation,
\beq \label{gauge} \phi \to \phi e^{i\chi}, \quad A_\mu \to A_\mu
- \partial_{\mu} \chi, \eeq for any arbitrary $C^\infty$ function
$\chi$.

This model was first proposed and studied in \cite{jackiw0,
jackiw, jackiw2}, and sometimes has received the name of
Chern-Simons-Schr\"{o}dinger equation. The initial value problem,
as well as global existence and blow-up, has been addressed in
\cite{berge, huh, huh3} for the case $p=3$.

\

The existence of stationary states for \eqref{eq:e0} and general
$p>1$  has been studied recently in \cite{byeon} (with respect to that paper, our
notation interchanges the indices $1$ and $2$). By using the ansatz:
\begin{equation*}
\begin{array} {lll} \phi(t,x) = u(|x|) e^{i \omega t}, && A_0(x)= A_0(|x|),
\\
A_1(t,x)=\dis -\frac{x_2}{|x|^2}h(|x|), && A_2(t,x)= \dis \frac{x_1}{|x|^2}h(|x|),
\end{array}
\end{equation*}
in \cite{byeon} it is found that $u$ solves the equation:: \beq
\label{equation}   - \Delta u(x) + \left( \o + \xi+ \dis
\frac{h^2(|x|)}{|x|^2} +  \int_{|x|}^{+\infty} \frac{h(s)}{s}
u^2(s)\, ds   \right)  u(x)  =|u(x)|^{p-1}u(x), \quad x \in
\RD,\eeq
where
$$h(r)= \frac 12\int_0^r s u^2(s)  \, ds.$$
Here $\xi$ in $\R$ is an integration constant of $A_0$, which
takes the form:

$$A_0(r)= \xi + \int_r^{+\infty} \frac{h(s)}{s} u^2(s)\, ds.$$

Observe that \eqref{equation} is a nonlocal equation. Moreover, in
\cite{byeon} it is shown that \eqref{equation} is indeed the
Euler-Lagrange equation of the energy functional:
$$ I_{\o+\xi}:
H_r^1(\R^2) \to \R, $$
defined as
\begin{align*}
I_{\o+\xi}(u) & = \dis \frac 12 \int_{\R^2} \left(|\nabla u(x)|^2
+ (\o+ \xi) u^2(x) \right) \, dx \nonumber
\\
 & \quad+\dis \frac{1}{8} \int_{\R^2}
\frac{u^2(x)}{|x|^2}\left(\int_0^{|x|} s u^2(s) \, ds\right)^2 dx  -
\dis \frac{1}{p+1}  \int_{\R^2} |u(x)|^{p+1} \, dx.
\end{align*}
Here $H^1_r(\R^2)$ denotes the Sobolev space of radially symmetric
functions. It is important to observe that the energy functional
$I_{\o+\xi}$ presents a competition between the nonlocal term and the
local nonlinearity. The study of the behavior of the functional
under this competition is one of the main motivations of this
paper.

Given a stationary solution, and taking $\chi = c\, t$ in the
gauge invariance \eqref{gauge}, we obtain another stationary
solution; the functions $u(x)$, $A_1(x)$, $A_2(x)$ are preserved,
and
$$ \o \to \o+c, \quad A_0(x) \to A_0(x)-c $$

Therefore, the constant $\omega + \xi$ is a gauge invariant of the
stationary solutions of the problem. By the above discussion we
can take $\xi=0$ in what follows, that is,
 $$ \lim_{|x|\to + \infty} A_0(x)=0,$$
which was indeed assumed in \cite{berge, jackiw2}.

For $p>3$, it is shown in \cite{byeon} that $I_\o$ is unbounded
from below, so it exhibits a mountain-pass geometry. In a certain
sense, in this case the local nonlinearity dominates the
nonlocal term. However the existence of a solution is not so
direct, since for $p \in (3,5)$ the (PS) property is not known to
hold. This problem is bypassed in \cite{byeon} by using a
constrained minimization taking into account the Nehari and
Pohozaev identities, in the spirit of \cite{yo}. Moreover,
infinitely many solutions have been found in \cite{huh2} for $p>5$
(possibly sign-changing).

A special case in the above equation is $p=3$: in this case,
static solutions can be found by passing to a self-dual equation,
which leads to a Liouville equation that can be solved explicitly.
Those are the unique positive solutions, as proved in
\cite{byeon}. For more information on the self-dual equations, see \cite{dunne,
jackiw2, tar}.

In case $p\in(1,3)$, solutions are found in \cite{byeon} as
minimizers on a $L^2$ sphere. Therefore, the value $\o$ comes out
as a Lagrange multiplier, and it is not controlled. Moreover, the
global behavior of the energy functional $I_\o$ is not studied.

The main purpose of this paper is to study whether $I_\o$ is
bounded from below or not for $p\in (1,3)$. In this case, the
nonlocal term prevails over the local nonlinearity, in a certain
sense. As we shall see, the situation is quite rich and unexpected
a priori, and very different from the usual Nonlinear Schr\"{o}dinger
Equation. This situation differs also from the Schr\"{o}dinger-Poisson
problem (see \cite{yo}), which is another problem presenting
a competition between local and nonlocal nonlinearities.

We shall prove the existence of a threshold value $\o_0$ such that
$I_\o$ is bounded from below if $\o \ge \o_0$, and it is not for
$\o \in (0,\o_0)$. But, in our opinion, what is most surprising is
that $\o_0$ has an explicit expression,
namely:
\beq \label{w0 bis} \omega_0= \frac{3-p}{3+p}\
3^{\frac{p-1}{2(3-p)}}\ 2^{\frac{2}{3-p}} \left(
\frac{m^2(3+p)}{p-1}\right)^{-\frac{p-1}{2(3-p)}}, \eeq with
$$ m = \ir  \left ( \frac{2}{p+1} \cosh^2 \left (\frac{p-1}{2} r \right) \right)^{\frac{2}{1-p}} \, dr.$$

Let us give an idea of the proofs. It is not difficult to show
that $I_\o$ is coercive when the problem is posed on a bounded
domain. So, there exists a minimizer $u_n$ on the ball $B(0,n)$
with Dirichlet boundary conditions. To prove boundedness of $u_n$,
the problem is the possible loss of mass at infinity as $n \to
+\infty$. The core of our proofs is a detailed study of the
behavior of those masses. We are able to show that, if unbounded,
the sequence $u_n$ behaves as a soliton, if $u_n$ is interpreted
as a function of a single real variable. The proof uses a careful
study of the level sets of $u_n$, which take into account the
effect of the nonlocal term. Then, the energy functional $I_\o$
admits a natural approximation through a convenient limit
functional. Finally, the solutions of that limit functional, and
their energy, can be found explicitly, so we can find $\o_0$. See Section 2 for an heuristic explanation of the proof and a derivation of the limit functional.

Regarding the existence of solutions, a priori, the global
minimizer could correspond to the zero solution. And indeed
this is the case for large $\o$. Instead, we show that $\inf I_\o
< 0$ if $\o > \o_0$ is close to the threshold value. Therefore, the
global minimizer is not trivial, and corresponds to a positive
solution. The mountain pass theorem will provide the existence of
a second positive solution.

If $\o<\o_0$, $I_\o$ is unbounded from below, and hence the
geometric assumptions of the mountain-pass theorem are satisfied.
However, the boundedness of (PS) sequences seems to be a hard
question in this case. Solutions are found for almost all values
of $\o \in (0,\o_0)$, by using the well-known monotonicity trick
of Struwe \cite{struwe} (see also \cite{jeanjean}).

Our main results are the following:

\begin{theorem} \label{teo1} For $\o_0$ as given in  \eqref{w0 bis}, there holds:

\begin{enumerate}[label=(\roman*), ref=\roman*]
\item \label{en:<} if $\o\in (0,\o_0)$, then $I_\o$ is unbounded
from below;

\item \label{en:=} if $\o=\o_0$, then $I_{\o_0}$ is bounded from
below, not coercive and $\inf I_{\o_0}<0$;

\item  \label{en:>} if $\o>\o_0$, then $I_\o$ is bounded from
below and coercive.
\end{enumerate}
\end{theorem}

Regarding the existence of solutions, we obtain the following
result:

\begin{theorem} \label{teo2} Consider \eqref{equation} with $\xi=0$. There exist $\bar{\omega} > \tilde{\o}>\o_0$ such that:

\begin{enumerate}[label=(\roman*), ref=\roman*]
\item if $\o> \bar{\o}$, then \eqref{equation} has no solutions
different from zero;
\item if $\o\in (\o_0,\tilde{\o})$, then
\eqref{equation} admits at least two positive solutions: one of
them is a global minimizer for $I_\o$ and the other is a
mountain-pass solution;
\item for almost every $\o\in (0,\o_0)$
\eqref{equation} admits a positive solution.
\end{enumerate}
\end{theorem}

The rest of the paper is organized as follows. Section 2 is
devoted to some preliminary results. Moreover, we give a heuristic
presentation of our proofs, which motivates the definition of the
limit functional. This limit functional is studied in detail in
Section 3. Finally, in Section 4 we prove Theorems \ref{teo1} and
\ref{teo2}.

\subsection*{Acknowledgement}
This work has been partially carried out during a stay of A.P. in Granada. He would like to
express his deep gratitude to the Departamento de An\'{a}lisis
Matem\'{a}tico for the support and warm hospitality.

\section{Preliminaries}

Let us first fix some notations. We denote by $H_r^1(\R^2)$ the
Sobolev space of radially symmetric functions, and $\| \cdot \|$ its usual norm. Other norms, like Lebesgue norms, will be
indicated with a subscript. In particular, $\|\cdot\|_{H^1(\R)}$, $\|\cdot\|_{H^1(a,b)}$ are
used to indicate the norms of the Sobolev spaces of dimension $1$.
If nothing is specified, strong and weak convergence of sequences
of functions are assumed in the space $\H$.

In our estimates, we will frequently denote by $C>0$, $c>0$ fixed
constants, that may change from line to line, but are always
independent of the variable under consideration. We also use the
notations $O(1), o(1), O(\e), o(\e)$ to describe the asymptotic
behaviors of quantities in a standard way. Finally the letters
$x$, $y$ indicate two-dimensional variables and $r$, $s$ denote
one-dimensional variables.

\

Let us start with the following proposition, proved in
\cite{byeon}:

\begin{proposition} $I_\o$ is a $C^1$ functional, and its critical
points correspond to classical solutions of \eqref{equation}.
\end{proposition}

Next result deals with the behavior of $I_\o$ under weak limits in
$H^1_r(\R^2)$. Even if it is not explicitly stated in this form,
Proposition \ref{prop:weak} follows easily from \cite[Lemma~3.2]{byeon} and the compactness of the embedding $H^1_r(\R^2)
\hookrightarrow L^q(\R^2)$, $q \in(2,+\infty)$ (see
\cite{strauss}).

\begin{proposition} \label{prop:weak} If $u_n \weakto u$, then
\[ %\label{c(u)u}
\int_{\R^2} \frac{u_n^2(x)}{|x|^2}\left(\int_0^{|x|} s u_n^2(s) \, ds\right)^2 dx
\to  \int_{\R^2} \frac{u^2(x)}{|x|^2}\left(\int_0^{|x|} s u^2(s) \,
ds\right)^2 dx.
\]

In particular, $I_\o$ is weak lower semicontinuous. Moreover, if
$u_n \weakto u$ then $I_\o'(u_n)(\varphi) \to I_\o'(u)(\varphi)$
for all $\varphi\in H^1_r(\R^2)$.

\end{proposition}

To finish the account of preliminaries, we now state an inequality
which will prove to be fundamental in our analysis. This
inequality is proved in \cite{byeon}, where also the maximizers are
found.

\begin{proposition}
For any $u \in H^1_r(\R^2)$,
\beq \label{ineq} \int_{\R^2} |u(x)|^4 \, dx \le 2 \left
(\int_{\R^2} |\nabla u(x)|^2 \, dx \right )^{1/2} \left
(\int_{\R^2} \frac{u^2}{|x|^2} \left( \int_0^{|x|} s u^2(s)\, ds
\right )^2 dx \right)^{1/2}. \eeq

\end{proposition}

As commented in the introduction, this paper is concerned with
boundedness from below of $I_\o$. Let us give a rough idea of the
arguments of our proof. First of all, consider $u(r)$ a fixed
function, and define $u_\rho(r)= u(r-\rho)$. Let us now estimate
$I_\o(u_\rho)$ as $\rho \to +\infty$.
\begin{align*}
(2\pi)^{-1} I_\o(u_\rho) &= \dis \frac 12 \int_{-\rho}^{+\infty} (|u'|^2
+ \o u^2)(r+\rho) \, dr
\\
& \quad+\dis \frac{1}{8} \int_{-\rho}^{\infty}
\frac{u^2(r)}{r+\rho}\left(\int_{-\rho}^r (s+\rho) u^2(s) \,
ds\right)^2 dr  - \dis \frac{1}{p+1}  \int_{-\rho}^{\infty}
|u|^{p+1}(r+\rho) \, dr.
\end{align*}

We estimate the above expression by simply replacing the
expressions $(r+\rho)$, $(s+\rho)$ with the constant $\rho$:
$$(2\pi)^{-1} I_\o(u)  $$$$\sim  \dis \rho \left [ \frac 12 \ir (|u|'^2
+ \o u^2) \, dr
 +\dis \frac{1}{8} \ir
u^2(r) \left(\int_{-\infty}^r  u^2(s) \, ds\right)^2 dr  - \dis
\frac{1}{p+1}  \ir |u|^{p+1} \, dr \right ]$$
$$ =\rho \left [ \frac 12 \ir (|u|'^2
+ \o u^2) \, dr
 +\dis \frac{1}{24} \left ( \ir
u^2 dr \right)^3  - \dis \frac{1}{p+1}  \ir |u|^{p+1} \, dr \right
]. $$

This estimate will be made rigorous in Lemma \ref{le:asin}.
Therefore, it is natural to define the limit functional $J_\o:
H^1(\R) \to \R$,
$$J_\o(u)= \frac 12 \ir \left(|u'|^2
+ \o u^2\right)  dr
 +\dis \frac{1}{24} \left ( \ir
u^2 dr \right)^3  - \dis \frac{1}{p+1}  \ir |u|^{p+1} \, dr.
$$

As a consequence of the above argument, if $J_\o$ attains negative
values, then $I_\o$ will be unbounded from below.

\medskip

The reverse is also true, but the proof is more delicate. We will
show that if $u_n$ is unbounded in $H_r^1(\R^2)$ and $I_\o(u_n)$
is bounded from above, then somehow $u_n$ contains a certain mass
spreading to infinity, as $u_\rho$ does. This will be made
explicit in Proposition \ref{prop-fund}. But this will lead us to
a contradiction if $J_\o$ is positive on that mass. The proof of
this argument is however far from trivial, and is the core of this
paper.

Summing up, we are able to relate $I_\o$ with the limit functional
$J_\o$ in the following way:
$$ \inf I_\o > -\infty \ \Leftrightarrow \ \inf J_\o =0.$$
Moreover this characterization will give us the threshold value for
$\o$, since the critical points of $J_\o$ can be found explicitly,
as will be shown in next Section.

\section{The limit problem}

In this section we deal with the limit functional $J_\o: H^1(\R)
\to \R$,
\beq \label{def J} J_{\omega}(u)= \frac 1 2  \ir \left ( |u'|^2 +
\omega u^2  \right )  dr + \frac{1}{24} \left ( \ir u^2 \, dr
\right)^3 - \frac{1}{p+1} \ir |u|^{p+1}\, dr. \eeq

Clearly, the Euler-Lagrange equation of \eqref{def J} is the
following problem: \beq \label{limit}
 -u'' + \omega u + \frac 1 4 \left ( \ir
u^2(s)\, ds \right )^2 u = |u|^{p-1} u, \quad \hbox{in }\R. \eeq

As we shall see later, we will find the explicit solutions of
\eqref{limit} later. But, first, let us study it from a
variational point of view: this study will give us some further
information on the solutions.

\

Before going on, we need a technical result, which is stated in
next lemma. We think that such result must be well-known, but we
have not been able to find a explicit reference.

\begin{lemma} \label{lemmino} Let $u_n \in H^1(\R)$ a sequence of even non-negative functions
which are decreasing in $r>0$, and assume that $u_n \weakto u_0$ weakly in $H^1(\R)$.
Then $u_0$ is also even, non-negative and decreasing in $r>0$, and
$u_n \to u_0$ in $L^q(\R)$ for any $q \in (2,+\infty)$.
\end{lemma}

\begin{proof}
Observe that the set $A=\{ u \in H^1(\R) \ \mbox{nonnegative, even
and decreasing in r}>0\}$ is a closed and convex subset of
$H^1(\R)$. As a consequence, $u_0 \in A$.

Then, for any $r\in \R$, $r\neq 0$,
$$
C \ge \left|\int_0^r u_n^2(s) \, ds\right| \ge u_n^2(r) |r| \Rightarrow u_n(r)
\le \frac{C}{\sqrt{|r|}},
$$
and the same estimate works for $u_0$. With this inequality, we can estimate:
\begin{align*}
 \ir |u_n-u_0|^q \, dr &\le \int_{-R}^R |u_n-u_0|^q \, dr + 2 C \int_{|r|>R} r^{-q/2}\, dr
 \\
&= \int_{-R}^R |u_n-u_0|^q \, dr + 4 C \frac{2}{2-q}
R^{\frac{2-q}{2}}.
\end{align*}
Take into account that, by Rellich-Kondrachov Theorem, $u_n \to
u_0 $ in $L^q(-R, R)$ for any $R>0$ fixed. Then, the above
inequality implies that $u_n \to u_0 $ in $L^q(\R)$.

\end{proof}

Some of the properties of the functional $J_{\o}$ are discussed
below:

\begin{proposition} \label{pr: limit1} Consider the functional $J_{\omega}$ with $p\in (1,3)$ and
$\omega > 0$. The following properties hold:

\begin{enumerate}
\item[a)] $J_{\o}$ is coercive and attains its infimum.

\item[b)] $0$ is a local minimum of $J_{\o}$. Indeed,
there exists $r_0>0$ with the following property:

\medskip \noindent for any $r \in (0, r_0)$, there exists $\alpha>0$ satisfying
that $J_{\o}(u)>\alpha$, for any $u \in H^1(\R)$ with $\|u\|_{H^1(\R)}=r$.

\item[c)] There exists $\omega_0>0$ such that $\min J_{\omega} <0$
if and only if $\omega \in [0, \omega_0)$.
\end{enumerate}

\end{proposition}

\begin{proof}

{\bf Proof of a)} To prove coercivity, we use Gagliardo-Nirenberg
inequality:
$$ \|u\|_{L^4} \le C \| u' \|_{L^2}^{1/4} \|u\|_{L^2}^{3/4}. $$
Hence
$$ \ir u^4\, dr \le \frac C 2 \left [ \ir |u' |^2\, dr + \left ( \ir u^2\, dr \right )^3 \right ]. $$
Then,
\beq J_{\o}(u) \label{pu} \ge \frac 1 4  \ir |u'|^2 \, dr +
\frac{1}{48} \left ( \ir u^2 \, dr \right)^3  + c \ir u^4\,
dr- \frac{1}{p+1} \ir |u|^{p+1}\, dr. \eeq

Observe that for any $C>0$ we can choose $D>0$ so that $t^3 \ge C
t - D$ for every $t \ge 0$. Applying this with $t = \ir u^2\,
dr$ into \eqref{pu}, and renaming $C$, we obtain:
$$J_{\o}(u) \ge \frac 1 4 \ir | u'|^2 \, dr + \ir \Big( C u^2 + c  u^4 -
\frac{1}{p+1} |u|^{p+1} \Big) dr - D.$$

Now, it suffices to take $C$ so that the function $C u^2+ c  u^4
-\frac{1}{p+1} |u|^{p+1} \ge 0$ for any $u \in \R$.

Take now $u_n$ such that $J_\o(u_n) \to \inf J_\o$. From the
coercivity, it follows that $u_n$ is bounded. Consider now the
sequence $v_n=|u_n|^*$ of non-negative symmetrized functions.
Clearly, $v_n$ is also bounded, and it is easy to observe that
$\inf J_\o \le J_\o(v_n) \le J_\o(u_n) \to \inf J_\o$.

Assume, passing to a subsequence, that $v_n \weakto v$ weakly in $H^1(\R)$. By Lemma
\ref{lemmino}, $v_n \to v$ in $L^{p+1}(\R)$. The weak lower
semicontinuity of the norm allows us to conclude that $u$ is a
minimizer of $J_{\o}$.

\medskip {\bf Proof of b)} This is quite standard. Indeed, by using Sobolev inequality,
$$J_\o(u) \ge \frac{1}{2} \min \{1,\omega \} \| u\|_{H^1(\R)} ^2 - C \|u\|_{H^1(\R)}^{p+1}.$$

\medskip {\bf Proof of c)} Let us define the map $\phi: [0, + \infty) \to
\R$, $\phi(\omega) = \min J_{\omega}$. It is easy to check that
$\phi$ is increasing and continuous. Moreover, $\phi(\omega) \le
0 $ for all $\omega$ (observe that  $J_{\o}(0)=0$).

We claim that $\phi(\omega)=0$ for large $\omega$. Indeed, by the
same arguments of the proof of {\bf a)}:
$$J_{\o}(u) \ge  \ir \left(\frac{\omega}{2} u^2 + c u^4 -
\frac{1}{p+1} |u|^{p+1}\right) dr.$$

For $\o$ sufficiently large, $\frac{\omega}{2} u^2 + c u^4 -
\frac{1}{p+1} |u|^{p+1}\ge 0$ for any $u\in \R$. Then $J_{\o}(u) \ge
0$ for any $u \in H^1(\R)$, proving the claim.

\bigskip

We now show that $\phi(0)<0$. On that purpose, fix $u \in H^1(\R)$
and define $u_{\lambda}(r)= \lambda^{\frac{2}{p-1}} u(\lambda r)$.
There holds:
$$J_0(u_\l)= \frac 1 2  \l^{\frac{p+3}{p-1}}  \ir |u'|^2 \, dr +
\frac{1}{24} \l^{\frac{3(5-p)}{p-1}} \left ( \ir u^2 \, dr
\right)^3 - \frac{1}{p+1} \l^{\frac{p+3}{p-1}} \ir |u|^{p+1}\,
dr.$$
Therefore, for $\l$ sufficiently small, $J_0(u_\l)$ has the sign
of the term
$$
\frac 1 2  \ir |u'|^2 \, dr - \frac{1}{p+1} \ir |u|^{p+1}\, dr.
$$
It suffices to take $u$ such that this
quantity is negative to conclude.

So, we can define $\omega_0= \min \{ \omega \ge 0: \ \phi(\omega)
= 0\}>0$.

\end{proof}

As a consequence of the previous result, for $\omega \in [0,
\omega_0)$ there exists a nontrivial solution for \eqref{limit},
which corresponds to a global minimum of $J_\o$. As announced in
the introduction, the expression for $\o_0$ will found later on.

\

We now pass to finding the explicit solutions of problem
\eqref{limit}. For any $k>0$ we denote by $w_k \in H^1(\R)$ the
unique positive radial solution of:
\beq \label{wk}
-w_k'' + k w_k = w_k^p, \quad \hbox{in } \R.
\eeq

Let us state some well-known properties of this equation. First,
the Hamiltonian of $w_k$ is equal to $0$, that is,

\beq \label{hamiltonian}
 - \frac 1 2 |w_k'(r)|^2 + \frac k 2
w_k^2(r) - \frac{1}{p+1} w_k^{p+1}(r)=0, \ \hbox{ for all } r \in
\R. \eeq

It is also known that any solution of \eqref{wk} is of the form
$u(x)= \pm w_k(x-y)$, for some $y\in \R$. Moreover, \beq
\label{scaling} w_k(r)= k^{\frac{1}{p-1}}w_1(\sqrt{k}r), \quad
\hbox{where} \quad w_1(r)= \left ( \frac{2}{p+1} \cosh^2 \left
(\frac{p-1}{2} r \right) \right)^{\frac{1}{1-p}}. \eeq

In what follows we define
\[  \label{defm}
m= \ir w_1^2\,dr.
\]

The following relations are also well known, and can be deduced
from \eqref{hamiltonian}: \beq \label{relations} \ir |w_1'|^2\, dr
= \frac{p-1}{p+3} m, \qquad \ir w_1^{p+1}\, dr =
\frac{2(p+1)}{p+3} m. \eeq

\begin{proposition} \label{explicit} Let us consider the equation:
\beq \label{eq-k} k= \o + \frac 1 4 m^2 k^{\frac{5-p}{p-1}},\ k>0.
\eeq

Then, $u$ is a nontrivial solution of \eqref{limit} if and only if
$u(r)= w_k(r-\xi)$ for some $\xi \in \R$ and $k$ a root of
\eqref{eq-k}.

Define:
\beq \label{w1} \omega_1= \left( \frac{(5-p)m^2}{4(p-1)}  \right
)^{-\frac{p-1}{2(3-p)}} - \frac{m^2}{4} \left(
\frac{(5-p)m^2}{4(p-1)}  \right )^{-\frac{(5-p)}{2(3-p)}}.\eeq

The following holds:
\begin{enumerate}
\item if $\o > \o_1$, equation \eqref{eq-k} has no solution and
there is no nontrivial solution of \eqref{limit};
\item if $\o =\o_1$, equation \eqref{eq-k} has only one solution
$k_0$ and $w_{k_0}(r)$ is the only non-trivial solution of
\eqref{limit} (apart from translations);
\item if $\o \in (0, \o_1)$, equation \eqref{eq-k} has two
solutions $k_1(\o)<k_2(\o)$ and $w_{k_1}(r),w_{k_2}(r)$
are the only two non-trivial solutions of \eqref{limit} (apart
from translations).

\end{enumerate}

\end{proposition}

\begin{proof}

Let $u$ be a nontrivial solution of \eqref{limit}, and define $k=
\omega + \frac 1 4 \left( \ir u^2\, dr\right )^2$. Then, $u$ is
a solution of $-u'' + k u = u^p$, so $u(r)= w_k(r-\xi)$ for some $\xi
\in \R$. By using \eqref{scaling}, we obtain:
$$ k= \omega + \frac 1 4 \left( \ir w_k^2(r)\, dr\right )^2=
\omega + \frac 1 4 k^{\frac{4}{p-1}} \left( \ir w_1^2(\sqrt{k}r)\,
dr\right )^2.$$
A change of variables leads us to equation \eqref{eq-k}.

Moreover,
$$ 1<p<3 \Rightarrow \frac{5-p}{p-1}>1.$$
Therefore, the function $(0,+\infty) \ni k \mapsto
k^{\frac{5-p}{p-1}}$ is convex. Therefore, there exists
$\omega_1>0$ with the properties indicated.

In order to get the exact value of $\o_1$, observe that the
function $k \mapsto  \o_1 + \frac 1 4 m^2 k^{\frac{5-p}{p-1}}-k$
has a degenerate $0$. Then, $\o_1$ solves the system:
$$ \left \{ \begin{array}{l}  \o + \displaystyle \frac{m^2}{4} k^{\frac{5-p}{p-1}}=k,
\\ \displaystyle \frac{5-p}{4(p-1)} m^2 k^{\frac{5-p}{p-1}-1}=1.
\end{array} \right.$$

From this one obtains formula \eqref{w1}.
\end{proof}

In our next result, we obtain information from Proposition
\ref{explicit}.

\begin{proposition} \label{gluing} Let $\o_0, \ \o_1$ be the values defined in
Propositions \ref{pr: limit1} and \ref{explicit} . Then:

\begin{enumerate}
\item $\o_0 < \o_1$, and $\o_0$ has the expression: \beq
\label{w0} \omega_0= \frac{3-p}{3+p}\ 3^{\frac{p-1}{2(3-p)}}\
2^{\frac{2}{3-p}} \left(
\frac{m^2(3+p)}{p-1}\right)^{-\frac{p-1}{2(3-p)}}, \eeq
where $m$ is as in \eqref{defm}.
\item For any $\o \in (0,\o_1)$, $J_\o(w_{k_1})>J_\o(w_{k_2})$. In
particular, for any $\o \in (0,\o_0)$, $w_{k_2}$ is a global
minimizer of $J_{\o}$.

\end{enumerate}

\end{proposition}

\begin{proof}

We consider the energy functional $J_{\o}$ evaluated on the curve $k
\mapsto w_k$. In the computations that follow we use
\eqref{scaling} and change of variables. We have
\begin{align*}
\psi(k)&:= J_{\o}(w_k)= \frac{k^{\frac{3+p}{2(p-1)}}}{2}  \ir |w_1'(r)|^2\, dr  + \o \frac{k^{\frac{5-p}{2(p-1)}}}{2}  \ir
 w_1^2(r)  \, dr
 \\
&\quad + \frac{k^{\frac{3(5-p)}{2(p-1)}}}{24}
\left ( \ir w_1^2(r) \, dr \right)^3 -
\frac{k^{\frac{3+p}{2(p-1)}}}{p+1} \ir |w_1(r)|^{p+1}\, dr.
\end{align*}
Plugging \eqref{relations} into that expression,
$$\psi(k)= m \left [ \frac{p-5}{2(3+p)}k^{\frac{3+p}{2(p-1)}}+ \frac{\omega}{2} k^{\frac{5-p}{2(p-1)}} +
\frac{m^2}{24} k^{\frac{3(5-p)}{2(p-1)}}\right ].  $$ Then:
$$\frac{d}{dk} \psi(k) = m\  k^{\frac{7-3p}{2(p-1)}} \frac{5-p}{4(p-1)} \left
[  -k + \o + \frac 1 4 m^2 k^{\frac{5-p}{p-1}} \right].$$
In particular, the roots of \eqref{eq-k} are exactly the critical
points of $\psi$. Observe that:
$$ \frac{5-p}{2(p-1)}< \frac{3+p}{2(p-1)}< \frac{3(5-p)}{2(p-1)}.$$

Then $\psi$ is increasing near $0$ (for $\o>0$) and near infinity.
Therefore, for $\o \in (0, \o_1)$, its first root corresponds to a
local maximum of $\psi$ and the second one to a local minimum, so
$J(w_{k_1})>J(w_{k_2})$. Take now $\o \in (0,\o_0)$. Since in this
case the minimizer is nontrivial, it must correspond to $w_{k_2}$.
Moreover, $\o_0<\o_1$.

\medskip In order to get the value of $\o_0$, observe that
$J_{\o_0}(w_{k_2})=0$. Therefore, $\o_0>0$ solves:
$$ \left \{ \begin{array}{l}  \o + \frac 1 4 m^2 k^{\frac{5-p}{p-1}}=k,\\
\frac{p-5}{2(3+p)}k^{\frac{3+p}{2(p-1)}}+ \frac{\omega}{2}
k^{\frac{5-p}{2(p-1)}} + \frac{m^2}{24}
k^{\frac{3(5-p)}{2(p-1)}}=0.\end{array}\right.$$
From there, expression \eqref{w0} follows.

\end{proof}

\begin{remark}
Observe that the map $\psi$ defined in the proof of Proposition
\ref{gluing} gives us a quite clear interpretation of the
functional $J_{\o}$. Indeed, $k$ is a critical point of $\psi$ if and
only if $w_k$ is a critical point of $J_{\o}$. Moreover, the following holds.

\begin{enumerate}
\item If $\o>\o_1$, $\psi$ is positive and increasing without
critical points.

\item If $\o= \o_1$, $\psi$ is still positive and increasing, but
it has an inflection point at $k=k_0$.

\item If $\o \in (0,\o_1)$, $\psi$ has a local maximum and minimum
attained at $k_1$ and $k_2$, respectively.

\item If $\o=\o_0$, $\psi(k_2)=0$. Observe then, in this case, the
minimum of $J_{\o_0}$ is $0$, and is attained at $0$ and
$w_{k_2}$.

\item If $\o \in [0, \o_0)$, $\psi(k_2)<0$ and then $w_{k_2}$ is the
unique global minimizer, with $J_\o(w_{k_2})<0$.

\end{enumerate}

\end{remark}

\begin{remark} In general, we cannot obtain a more explicit expression of $m$ depending on
$p$, but it can be easily approximated by using some software. In
Figure 1 the maps $\o_0(p)$ and $\o_1(p)$ have been plotted.

\begin{figure}[h]
\centering \fbox{
\begin{minipage}[c]{110mm}
           \centering
        \resizebox{99mm}{66mm}{\includegraphics{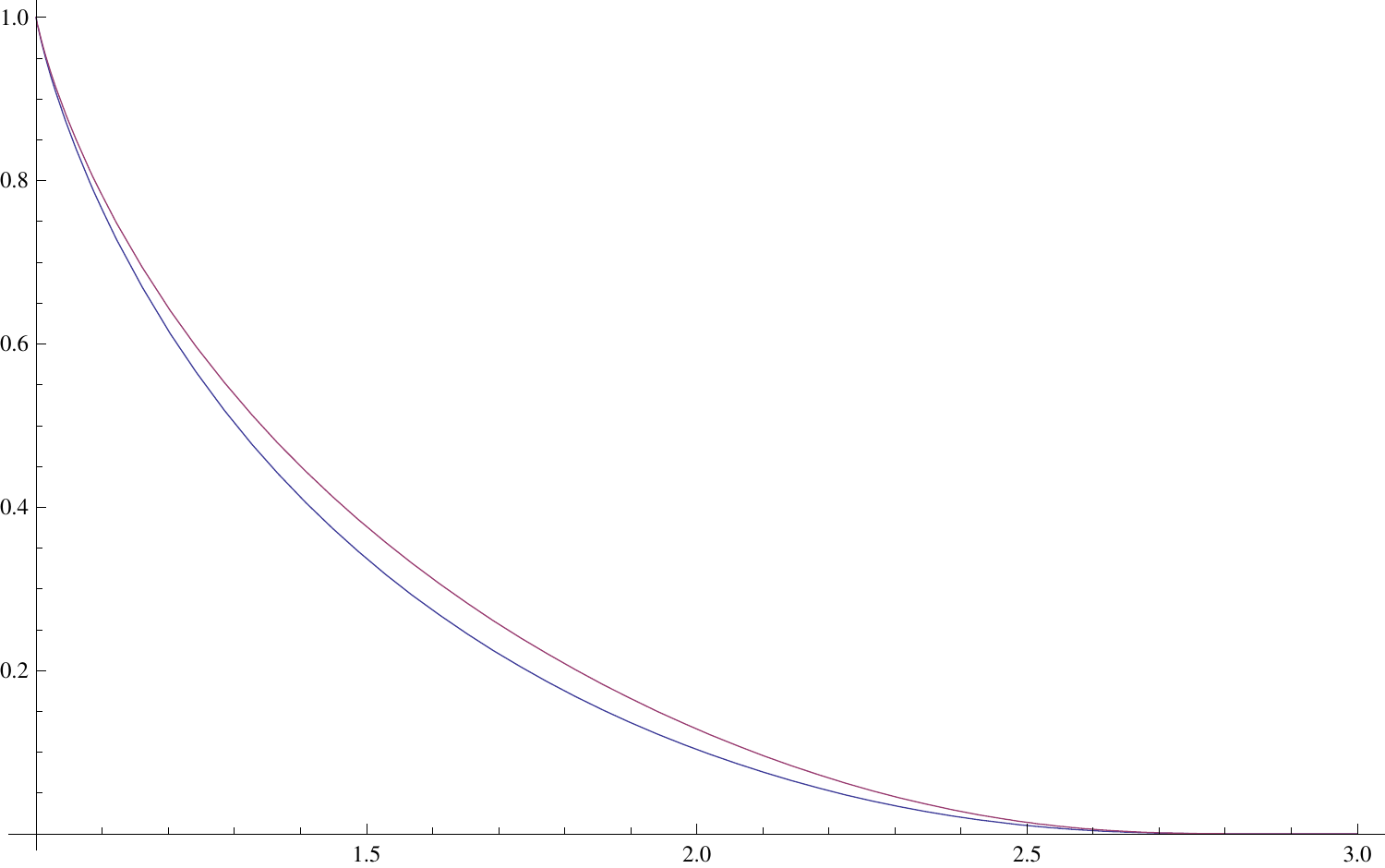}}
       \caption{\footnotesize The values $\o_0(p) <\o_1(p)$, for $p\in (1,3)$.}

         \end{minipage}
       }
\end{figure}

For some specific values of $p$, $m$ can be explicitly computed,
and hence $\o_0$ and $\o_1$. For instance, if $p=2$, $m=6$,
$\o_1=\frac{2}{9\sqrt{3}}$ and $\o_0=\frac{2}{5\sqrt{15}}$.

\end{remark}

We finish this section with a technical result that will be of use
later in the proof of Theorem \ref{teo1}.

\begin{proposition} \label{extra}
Assume $\omega \ge \o_0$, and $u_n \in H^1(\R)$ such that $J_{\o}(u_n)
\to 0$. There holds
\begin{enumerate}
\item if $\o > \o_0$, then $u_n \to 0$ in $H^1(\R)$;
\item if $\o=\o_0$, then, up to a subsequence, either $u_n \to 0$ or $u_n(\cdot-x_n) \to w_{k_2}$ in $H^1(\R)$, for some sequence $x_n \in \R$.

\end{enumerate}
\end{proposition}

\begin{proof}

Since $J_{\o}$ is coercive, we have that $u_n$ is bounded. If $u_n \to
0$ in $H^1(\R)$, we are done. Otherwise, we have that:
$$ o_n(1)= J_{\o}(u_n) \ge \frac 1 2 \ir \left(|u_n'(r)|^2 + \o u_n^2(r)\right) dr - \frac{1}{p+1} \ir |u_n(r)|^{p+1}\, dr.$$
Then, $u_n \nrightarrow 0$ in $L^{p+1}(\R)$. We can apply
concentration-compactness lemma (see \cite[Lemma I.1]{lions}), and
there exists $\xi_n \in \R$ such that $\int_{\xi_n-1}^{\xi_n+1} u_n^2
\ge \e >0$. Therefore, $\tilde{u}_n(r)= u_n(r-\xi_n) \weakto u \neq
0$ weakly in $H^1(\R)$. Define $v_n= \tilde{u}_n -u$, which clearly converges weakly
to $0$ in $H^1(\R)$.

\medskip

{\bf Step 1: $v_n \to 0$ in $L^{2}(\R)$.}

We just compute
\begin{align*}
o_n(1)&= J_{\o}(u_n)= J_{\o}(\tilde{u_n})= J_{\o}(v_n+u)
\\
& = \frac 1 2 \ir
\left(|v_n'|^2 + |u'|^2 + 2 v_n' u ' \right)  dr + \frac{\o}{2} \ir
\left(v_n^2 + u^2+ 2 v_n u \right) dr
\\
&\quad+ \frac 1 8 \left [
\left ( \ir v_n^2 \, dr \right )^3 + \left ( \ir u^2 \, dr
\right )^3  + 3 \left ( \ir v_n^2 \, dr \right )^2\left ( \ir
u^2 \, dr \right ) \right.
\\
& \quad\left.  +  3 \left ( \ir v_n^2 \, dr \right )\left ( \ir u^2 \, dr \right )^2 \right
] - \frac{1}{p+1} \ir |v_n+u|^{p+1}\, dr
+o_n(1).
\end{align*}
Here the mixed products converge to zero, since $v_n \weakto 0$.
Passing to a subsequence, we can assume that $v_n \to 0$ almost
everywhere. Then, the well-known Brezis-Lieb lemma
(\cite{brezis-lieb}) implies that
$$ \ir |v_n+u|^{p+1}\, dr - \ir (|v_n|^{p+1} + |u|^{p+1})\, dr  \to 0. $$
Then,
\begin{align*}
o_n(1)&= J_{\o}(u_n) = J_{\o}(v_n)+ J_{\o}(u) + \frac{3}{8} \left [ \left (
\ir v_n^2 \, dr \right )^2\left ( \ir u^2 \, dr \right )
\right.
\\
&\quad \left. +  \left ( \ir v_n^2 \, dr \right )\left (
\ir u^2 \, dr \right )^2 \right ] + o_n(1).
\end{align*}

It is here that the assumption $\o \ge \o_0$ is crucial. Indeed,
it implies that $J_{\o}(v_n) \ge 0$, $J_{\o}(u)\ge 0$. Recall that $u \neq
0$, to conclude the proof of Step 1.

\medskip {\bf Step 2: Conclusion.}

By interpolation,
$$ \| v_n\|_{L^{p+1}} \le  \| v_n\|_{L^{2}}^{\a}  \|
v_n\|_{L^{p+2}}^{1-\a},$$ for some $\a \in (0,1)$. Since $v_n$ is
bounded in $H^1(\R)$, then all norms above are bounded. Then, by
Step 1, $\| v_n\|_{L^{p+1}} \to 0$. In other words, $\tilde{u}_n
\to u$ in $L^{p+1}(\R)$.

From this it is easy to conclude. Indeed,
\begin{align*}
o_n(1)&= J_{\o}(\tilde{u}_n) = \frac 1 2 \ir \left(|\tilde{u}_n'|^2 +
\o \tilde{u}_n^2 \right)  dr + \frac 1 8  \left ( \ir
\tilde{u}_n^2 \, dr \right )^3 - \frac{1}{p+1} \ir
|\tilde{u}_n|^{p+1}\,dr,
\\
0 &\le J_{\o}(u) =  \frac 1 2 \ir
\left(|u'|^2 + \o u^2\right) dr + \frac 1 8  \left ( \ir u^2 \, dr
\right )^3 - \frac{1}{p+1} \ir |u|^{p+1}\,dr.
\end{align*}
Then, $\|\tilde{u}_n\|_{H^1(\R)} \to \|u\|_{H^1(\R)}$. And this implies that
$\tilde{u}_n \to u$ in $H^1(\R)$, finishing the proof.

\end{proof}

\section{Proof of Theorems \ref{teo1}, \ref{teo2}}
\begin{lemma}\label{le:asin}
Let $U\in H^1(\R)$ be an even function which decays to zero exponentially at infinity, and define $U_\rho(r)= U(r-\rho)$. Then there exists $C>0$ such that:
\[
I_\o(U_\rho)=2 \pi \rho J_\o(U)-C+o_\rho(1).
\]
\end{lemma}

\begin{proof}
We have
\begin{align}
(2\pi)^{-1} I_\o(U_\rho) &= \frac 12\int_0^{+\infty}
\left(|U'_\rho|^2 +\o U_\rho^2 \right)r \, dr +\frac
18\int_0^{+\infty} \frac{U_\rho^2(r)}{r}\left(\int_0^r s
U_\rho^2(s) \, ds\right)^2 dr \nonumber
\\
&\quad-\frac{1}{p+1} \int_0^{+\infty} |U_\rho|^{p+1} r\, dr. \label{eq:asin}
\end{align}
Let us, first of all, evaluate the local terms. By the evenness and the exponential decay of $U$, we get
\begin{align}
\int_0^{+\infty}|U_\rho'|^2  r \, dr
&=\int_{-\infty}^{+\infty} |U'(r-\rho)|^2   (r-\rho) \, dr +\rho\int_{-\infty}^{+\infty}|U'(r-\rho)|^2  \, dr +o_\rho(1)\nonumber
\\
&=\rho \int_{-\infty}^{+\infty}|U'|^2  \, dr +o_\rho(1).\label{eq:asin1}
\end{align}
Analogously
\begin{align}
\int_0^{+\infty}U_\rho^2  r \, dr
&=\rho \int_{-\infty}^{+\infty} U^2  \, dr +o_\rho(1),\label{eq:asin2}
\\
\int_0^{+\infty}|U_\rho|^{p+1}  r \, dr
&=\rho \int_{-\infty}^{+\infty} |U|^{p+1}  \, dr +o_\rho(1).\label{eq:asin3}
\end{align}
For what concerns the nonlocal term, we have
\begin{align*}
&\int_0^{+\infty}  \frac{U_\rho^2(r)}{r}\left(\int_0^r s U_\rho^2(s) \, ds\right)^2 dr
-\rho\int_0^{+\infty}  U_\rho^2(r)\left(\int_0^r  U_\rho^2(s) \, ds\right)^2 dr
\\
&\qquad=\underset{(I)}{\underbrace{\int_0^{+\infty}  U_\rho^2(r) \left(\frac{1}{r}-\frac{1}{\rho}\right)\left(\int_0^r s U_\rho^2(s) \, ds\right)^2 dr}}
\\
&\qquad\quad+\underset{(II)}{\underbrace{\frac 1\rho\int_0^{+\infty}  U_\rho^2(r)\left[\left(\int_0^r  s U_\rho^2(s) \, ds\right)^2-\left(\int_0^r  \rho U_\rho^2(s) \, ds\right)^2 \right] dr}}.
\end{align*}
Let us study the term $(I)$:
\begin{align*}
(I)&=\int_{-\infty}^{+\infty}  U_\rho^2(r) \frac{\rho-r}{r\rho}\left(\int_{-\infty}^r s U_\rho^2(s) \, ds\right)^2 dr
+o_\rho(1)
\\
&=-\int_{-\infty}^{+\infty}  U^2(r) \frac{r}{(\rho+r)\rho}\left(\int_{-\infty}^r (s+\rho) U^2(s) \, ds\right)^2 dr
+o_\rho(1)
\\
&=\int_{0}^{+\infty}  U^2(r) \frac{r}{(\rho-r)\rho}\left(\int_{-\infty}^{-r} (s+\rho) U^2(s) \, ds\right)^2 dr
\\
&\quad-\int_{0}^{+\infty}  U^2(r) \frac{r}{(\rho+r)\rho}\left(\int_{-\infty}^r (s+\rho) U^2(s) \, ds\right)^2 dr
+o_\rho(1)
\\
&=\int_{0}^{+\infty}  U^2(r) \left(\frac{r}{(\rho-r)\rho}-\frac{r}{(\rho+r)\rho}\right)\left(\int_{-\infty}^{-r} (s+\rho) U^2(s) \, ds\right)^2 dr
\\
&\quad +\int_{0}^{+\infty}  U^2(r) \frac{r}{(\rho+r)\rho}
\left[\left(\int_{-\infty}^{-r} (s+\rho) U^2(s) \, ds\right)^2-\left(\int_{-\infty}^r (s+\rho) U^2(s) \, ds\right)^2\right] dr
\\
&\quad+o_\rho(1)
\\
&= \frac{1}{\rho} \int_{0}^{+\infty}  U^2(r) \left(\frac{2r^2\rho^2}{(\rho-r)(\rho+r)}\right)
\left(\int_{-\infty}^{-r} \frac{s+\rho}{\rho} \ U^2(s) \, ds\right)^2 dr
\\
& \quad +\int_{0}^{+\infty}  U^2(r) \frac{r\rho}{(\rho+r)}
\left[\left(\int_{-\infty}^{-r} \frac{s+\rho}{\rho} \ U^2(s) \, ds\right)^2-
\left(\int_{-\infty}^r \frac{s+\rho}{\rho}\ U^2(s) \, ds\right)^2\right] dr
\\
&\quad+o_\rho(1). \end{align*}
We now pass to the limit by Lebesgue Theorem, and obtain:
\begin{align*}
(I)&= \int_{0}^{+\infty}  U^2(r) r
\left[\left(\int_{-\infty}^{-r}  U^2(s) \, ds\right)^2-\left(\int_{-\infty}^r  U^2(s) \, ds\right)^2\right] dr
+o_\rho(1) &
\\
&=-C_I +o_\rho(1).
\end{align*}
Let us study the term $(II)$:
\begin{align*}
(II)&=\frac 1\rho\int_0^{+\infty}  U_\rho^2(r)\left(\int_0^r  (s+\rho) U_\rho^2(s) \, ds\right)\left(\int_0^r  (s-\rho) U_\rho^2(s) \, ds\right)dr
\\
&=\frac 1\rho\int_{-\infty}^{+\infty}  U_\rho^2(r)\left(\int_{-\infty}^r  (s+\rho) U_\rho^2(s) \, ds\right)\left(\int_{-\infty}^r  (s-\rho) U_\rho^2(s) \, ds\right)dr
+o_\rho(1)
\\
&=\int_{-\infty}^{+\infty}  U^2(r)\left(\int_{-\infty}^r  \frac{s+2\rho}{\rho} U^2(s) \, ds\right)\left(\int_{-\infty}^r  s U^2(s) \, ds\right)dr
+o_\rho(1). \end{align*}
Again by Lebesgue Theorem,
\begin{align*}
(II)&= 2\int_{-\infty}^{+\infty}  U^2(r)\left(\int_{-\infty}^r   U^2(s) \, ds\right)\left(\int_{-\infty}^r  s U^2(s) \, ds\right)dr
+o_\rho(1)
\\
&=-C_{II} +o_\rho(1).
\end{align*}
Observe that the above expression is negative since the function $r \mapsto \int_{-\infty}^r s U^2(s) \, ds$ is negative.
Therefore, denoting by $C=C_I+C_{II}>0$, we have
\beq\label{eq:asin4}
\int_0^{+\infty}  \frac{U_\rho^2(r)}{r}\left(\int_0^r s U_\rho^2(s) \, ds\right)^2 dr
=\rho\int_0^{+\infty}  U_\rho^2(r)\left(\int_0^r  U_\rho^2(s) \, ds\right)^2 dr
-C +o_\rho(1).
\eeq
Hence the conclusion follows by \eqref{eq:asin}, \eqref{eq:asin1}, \eqref{eq:asin2}, \eqref{eq:asin3} and \eqref{eq:asin4}.
\end{proof}

In our next result we study the behavior of unbounded sequences with energy
bounded from above. This will be essential for the proof of Theorems \ref{teo1}, \ref{teo2}.

\begin{proposition} \label{prop-fund} Assume $\o>0$, and $u_n \in \Hr$ such that $\|u_n\| $ is unbounded but
$I_{\o}(u_n)$ is bounded from above. Then, there exists a subsequence
(still denoted by $u_n$) such that:
\begin{enumerate}

\item[i)] for all $\e>0$, $\displaystyle \int_{\e
\|u_n\|^2}^{+\infty} \big(|u_n'|^2 + u_n^2 \big) \, dr \le C$;

\item[ii)] there exists $\d \in (0,1)$ such that  $\displaystyle
\int_{\d \|u_n\|^2}^{\d^{-1} \|u_n\|^2} \big(| u_n'|^2 + u_n^2 \big )\, dr
\ge c>0$;

\item[iii)] $\|u_n\|_{L^2(\RD)} \to +\infty$.

\end{enumerate}

\end{proposition}

\begin{proof}

The beginning of the proof follows the ideas of \cite[Theorem 4.3]{yo}. The main difference is that here we cannot conclude directly
that $I_\o$ is bounded from below, and indeed this fact depends on
$\o$. The proof of Theorem \ref{teo1} will require much more work.

We start using inequality \eqref{ineq} and Cauchy-Schwartz
inequality to estimate $I_\o$:
\begin{align}
I_\o(u) &\ge \frac{\pi}{2}\int_0^{+\infty} \left(|u'|^2 +\o u^2
\right)r \, dr +\frac{\pi}{8}\int_0^{+\infty}
\frac{u^2(r)}{r}\left(\int_0^r s u^2(s) \, ds\right)^2 dr
\nonumber
\\
&\quad+ 2\pi \int_0^{+\infty}\left(\frac{\o}{4}u^2
+\frac{1}{8}u^4-\frac{1}{p+1}  |u|^{p+1}\right) r\, dr.
\label{eq:If}
\end{align}
Define
\[
f:\R_+\to \R,\qquad f(t)= \frac{\o}{4}t^2 +\frac{1}{8}t^4-\frac{1}{p+1} t^{p+1}.
\]
Then, the set $\{t > 0: f(t) < 0\}$ is of the form $(\a,\b)$,
where $\a,\b$ are positive constants depending only on $p, \o$.
Moreover, we denote by $-c_0=\min f<0$.

For each function $u_n$, we define:
 \begin{equation*} %\label{defA}
 A_n=\{x\in \RD : u_n(x)\in (\a,\b)\},\ \rho_n= \sup\{|x|: x\in A_n\}.
 \end{equation*}
\\
With these definitions, we can rewrite \eqref{eq:If} in the form
\begin{equation} \label{eq:s0} I_\o(u_n) \ge
\frac{\pi}{2} \int_0^{+\infty} \left(|u_n'|^2 +\o u_n^2 \right)r
\, dr +\frac{\pi}{8}\int_0^{+\infty}
\frac{u_n^2(r)}{r}\left(\int_0^r s u_n^2(s) \, ds\right)^2 dr
-c_0|A_n|.
\end{equation}
In particular this implies that $|A_n|$ must diverge, and hence
$\rho_n$. This already proves (iii).

\medskip By Strauss Lemma \cite{strauss}, we have
\beq\label{eq:sl} \a\le u_n(\rho_n)\le
\frac{\|u_n\|}{\sqrt{\rho_n}}, \ \Rightarrow \|u_n\|^2\ge
\a^2\rho_n. \eeq

We now estimate the nonlocal term. For that, define
\begin{equation} \label{Bn} B_n=A_n\cap B(0,\g_n), \mbox{ for }
\g_n\in (0,\rho_n) \mbox{ such that }|B_n|=\frac 1 2|A_n|.
\end{equation}

Then,
\begin{align}
\int_0^{+\infty}  \frac{u_n^2(r)}{r}\left(\int_0^r s u_n^2(s) \,
ds\right)^2 dr &\ge \frac{1}{4\pi^2} \int_{\g_n}^{+\infty}
\frac{u_n^2(r)}{r}\left(\int_{B_n}  u_n^2(x) \, dx\right)^2 dr
\nonumber
\\
&\ge c|A_n|^2\int_{\g_n}^{+\infty} \frac{u_n^2(r)}{r}\, dr \nonumber
\\
&\ge c|A_n|^2\int_{A_n\setminus B_n} \frac{u_n^2(x)}{|x|^2}\, dx \nonumber
\\
&\ge c\frac{|A_n|^2}{\rho_n^2}\int_{A_n\setminus B_n} u_n^2(x)\, dx \nonumber
\\
&\ge c\frac{|A_n|^3}{\rho_n^2}.\label{eq:Arho}
\end{align}
Hence, by \eqref{eq:If}, \eqref{eq:sl} and \eqref{eq:Arho}, we get
\[
I_\o(u_n)
\ge c \rho_n +c\frac{|A_n|^3}{\rho_n^2}-c_0|A_n|
=  \rho_n\left(c +c\frac{|A_n|^3}{\rho_n^3}-c_0\frac{|A_n|}{\rho_n}\right).
\]
Observe that  $t\mapsto c+ct^3-c_0 t$ is strictly positive near
zero and goes to $+\infty$, as $t\to +\infty$. Then we can assume,
passing to a subsequence, that $|A_n| \sim \rho_n$. In other
words, there exists $m>0$ such that $ \rho_n|A_n|^{-1} \to m$ as
$n\to +\infty$.
\\
Taking into account \eqref{eq:s0} we conclude that up to a
subsequence, $\|u_n\|^2 \sim \rho_n$. Moreover, for any fixed
$\e>0$, we have:
\[
C\rho_n \ge \|u_n\|_{L^2}^2  \ge \int_{\e \rho_n}^{+\infty}
u_n^2r \,dr \ge \e \rho_n \int_{\e \rho_n}^{+\infty} u_n^2 \,dr.
\]
An analogous estimate works also for $\int_{\e \rho_n}^{+\infty}
|u_n'|^2 dr$. This proves (i).
\\
\
\\
We now show that for some $\d>0$, $\|u_n\|_{H^1(\d \rho_n,\rho_n)}
\nrightarrow 0$, which implies assertion (ii).
\\
First, recall the definition of $B_n$ and $\g_n$ in \eqref{Bn}.
Then,
$$ \int_{\g_n}^{\rho_n} u_n^2(r) \, dr \ge \rho_n^{-1} \int_{\g_n}^{\rho_n} u_n^2(r) r\,
dr\ge \rho_n^{-1} \int_{A_n\setminus B_n} u_n^2(x) dx \ge
\rho_n^{-1} |A_n\setminus B_n| \a^2>c>0.$$
To conclude it suffices to show that $\g_n \sim \rho_n$. Indeed,
by repeating the estimate \eqref{eq:Arho} with $A_n$ replaced by
$B_n$, we infer
\[
I_\o(u_n)
\ge c \rho_n +c\frac{|A_n|^3}{\g_n^2}-c_0|A_n|
=  \g_n\left(c \frac{\rho_n}{\g_n} +c\frac{|A_n|^3}{\g_n^3}-c_0\frac{|A_n|}{\g_n}\right).
\]
And we are done since $I_{\o}(u_n)$ is bounded from above.

\end{proof}

\begin{proof}[Proof of Theorem \ref{teo1}]

If $\o\in (0,\o_0)$, then $J_\o(w_{k_2})<0$ (see Proposition
\ref{pr: limit1}): applying Lemma \ref{le:asin} to $U= w_{k_2}$ we
conclude assertion (i).

\medskip
We now prove $(ii)$ and $(iii)$. Let us denote by $H^1_{0,r
}(B(0,R))$ the Sobolev space of radial functions with zero
boundary value. Given any $n \in \N$, Proposition \ref{prop-fund}
implies that $I_{\o}|_{H^1_{0,r }(B(0,n))}$ is coercive (indeed, this
is an immediate consequence of \eqref{eq:If}). So, there exists
$u_n$ a minimizer for $I_{\o}|_{H^1_{0, r}(B(0,n))}$. Moreover,
$$I_{\o}(u_n) \to \inf I_{\o}, \mbox{ as } n \to +\infty.$$

If $u_n$ is bounded, then $I_{\o}(u_n) $ is also bounded and
therefore $\inf I_\o$ is finite. In what follows we assume that
$u_n$ is an unbounded sequence. Then, the sequence $u_n$ satisfies
the hypotheses of Proposition \ref{prop-fund}. Let $\d>0$  be
given by that proposition.

The proof will be divided in several steps.

\

{\bf Step 1:} $\dis \int_{\frac{\delta}{2}
\|u_n\|^2}^{\frac{2}{\delta} \|u_n\|^2} |u_n|^{p+1}\, dr
\nrightarrow 0$.

\

 For each $n \in \N$, we can choose $x_n$, $y_n$: $$ \frac{\d}{2} \|u_n\|^2 < x_n<\d \|u_n\|^2-1,\ \  \d^{-1} \|u_n\|^2+1<y_n< 2\d^{-1} \|u_n\|^2$$ such that
\[
\|u_n\|_{H^1(x_n,x_n+1)}+\|u_n\|_{H^1(y_n,y_n+1)} \le \frac{C}{\|u_n\|}.
\]
Observe that if $\d^{-1} \|u_n\|^2 \ge n$, the choice of $y_n$ can be arbitrary, but it is unnecessary.
Take $\phi_n:[0,+\infty] \to [0,1]$ be a $C^{\infty}$-function such that
\[
\phi_n(r)=\left\{
\begin{array}{ll}
0, & \hbox{if }r\le x_n,
\\
1, & \hbox{if }  x_n+1\le r \le y_n,
\\
0, & \hbox{if } r\ge y_n+1.
\end{array} \ \ |\phi_n'(r)|\le 2.
\right.
\]
We have
\begin{align*}
0&=I_\o'(u_n)[\phi_n u_n] \ge 2\pi \int_{x_n}^{y_n} \left(|u'_n|^2 +\o
u_n^2\right)r \, dr
 - 2\pi \int_{x_n}^{y_n} |u_n|^{p+1}r \, dr +O(1)
\\
& \ge \|u_n\|^2 \left ( \frac{\d}{2} \int_{x_n}^{y_n} \left(|u'_n|^2 +\o u_n^2\right)  dr
 - \frac{2}{\d} \int_{x_n}^{y_n} |u_n|^{p+1}r \, dr \right)
+O(1).
\end{align*}
This, together with the fact that $\|u_n\|_{H^1(x_n, y_n)}$ does
not tend to zero, allows us to conclude the proof of Step 1.

\

{\bf Step 2:} Exponential decay.

\

At this point we can apply the concentration-compactness principle
(see \cite[Lemma 1.1]{lions}); there exists $\s>0$ such that
\[
\sup_{\xi\in [x_n,\ y_n]}\int_{\xi-1}^{\xi+1}u_n^2\, dr\ge 2\s>0.
\]
Let us define the set:

\beq \label{def-D_n} D_n=\left\{\xi \in [x_n,\ y_n] :
\int_{\xi-1}^{\xi+1} \left(|u_n'|^2+ u_n^2\right) dr\ge \s\right\} \neq
\emptyset, \quad \hbox{and} \quad \xi_n=\max D_n. \eeq

By definition, $\int_{\zeta-1}^{\zeta+1}(|u_n'|^2+u_n^2)\, dr <  \s$ for all $\zeta > \xi_n$. By embedding of $H^1(\zeta-1,\zeta+1)$ in $L^{\infty}$, $0<u_n(\zeta) < C \s$ for any $\zeta > \xi_n$. From this we will get exponential decay of $u_n$. Indeed, $u_n$ is a solution of
\[ %\label{eq-ball}
 - u_n''(r) - \frac{u'(r)}{r} + \o u_n(r) + f_n(r) u_n(r)  = |u_n(r)|^{p-1}u_n(r),
 \]
with
$$  f_n(r)=  \frac{h_n^2(r)}{r^2} + \int_r^n \frac{h_n(s)}{s} u_n^2(s) \, ds, \ \ \  h_n(r)= \frac 12\int_0^r u_n^2(s) s \, ds.$$

It is important to observe that $0\le f_n(r) \le C$  for all $r>\d
\|u_n\|^2$. Then, by taking smaller $\s$, if necessary, we can
conclude that there exists $C>0$ such that
$$ |u_n(r)| <C {\rm exp}\left(- \sqrt{\o} (r-\xi_n)\right), \quad
\mbox{ for all } r>\xi_n. $$

The local $C^1$ regularity theory for the Laplace operator (see
\cite[Section 3.4]{gilbarg}) implies a similar estimate for
$u_n'(r)$. In other words,
\beq \label{decay1} |u_n(r)| + |u_n'(r)|<C {\rm exp}\left(- \sqrt{\o}
(r-\xi_n)\right), \quad \mbox{ for all } r>\xi_n. \eeq

\

{\bf Step 3:} Splitting of $I_\o(u_n)$.

Let us now choose another cut-off function; take $z_n$:
 $$ \xi_n - 3 \|u_n\| < z_n < \xi_n-2\|u_n\|, $$ such that
\[
\|u_n\|^2_{H^1(z_n, z_n+1)} \le \frac{C}{\|u_n\|}.
\]

Define $\psi_n:[0,+\infty] \to [0,1]$ be a smooth function such that
\beq \label{def-psi_n}
\psi_n(r)=\left\{
\begin{array}{ll}
0, & \hbox{if }r\le z_n,
\\
1, & \hbox{if }  r \ge z_n+1,
\end{array}  \ \ |\psi_n'(r)|\le 2.
\right.
\eeq

In what follows we want to estimate $I_{\o}(u_n)$ with $I_{\o}(\psi_n u_n)$ and $I_{\o}\left((1-\psi_n) u_n\right)$. Let us start evaluating the local terms.
\begin{align*}
\int_0^{n} |u'_n|^2 r\, dr
&=\int_0^{n} |(u_n \psi_n)'|^2 r\, dr
+\int_0^{n} \big|\big(u_n (1-\psi_n)\big)'\big|^2 r\, dr
+O(\|u_n\|),
\\
\int_0^{n} u_n^2 r\, dr
&=\int_0^{n} |u_n \psi_n|^2 r\, dr
+\int_0^{n} \big|u_n (1-\psi_n)\big|^2 r\, dr
+O(\|u_n\|),
\\
\int_0^{n} |u_n|^{p+1} r\, dr
&=\int_0^{n} |u_n\psi_n |^{p+1} r\, dr
+\int_0^{n} |u_n (1-\psi_n)|^{p+1} r\, dr
+O(\|u_n\|).
\end{align*}
Let us study now the nonlocal term.
\begin{align*}
&\int_0^n \frac{u_n^2(r)}{r}\left(\int_0^r s u_n^2(s)\, ds\right)^2 dr
=\int_0^n \frac{u_n^2(r)\psi_n^2(r)}{r}\left(\int_0^r s u_n^2(s)\psi_n^2(s)\, ds\right)^2 dr
\\
&\qquad+\int_0^n \frac{u_n^2(r)(1-\psi_n(r))^2}{r}\left(\int_0^r s u_n^2(s)(1-\psi_n(s))^2\, ds\right)^2 dr
\\
&\qquad+\underset{(I)}{\underbrace{\int_0^n \frac{u_n^2(r)\psi_n^2(r)}{r}\left(\int_0^r s u_n^2(s)(1-\psi_n(s))^2\, ds\right)^2 dr}}
\\
&\qquad+2\underset{(II)}{\underbrace{\int_0^n \frac{u_n^2(r)\psi_n^2(r)}{r}\left(\int_0^r s u_n^2(s)\psi_n^2(s)\, ds\right)\left(\int_0^r s u_n^2(s)(1-\psi_n(s))^2\, ds\right) dr}}
\\
&\qquad+O(\|u_n\|).
\end{align*}
We now estimate:
$$ (I) \ge 0,$$
\begin{align*}
(II)&=\int_{z_n}^n \frac{u_n^2(r)\psi_n^2(r)}{r}\left(\int_{z_n}^r s u_n^2(s)\psi_n^2(s)\, ds\right)\left(\int_0^{z_n+1} s u_n^2(s)(1-\psi_n(s))^2\, ds\right) dr
\\
&\quad+ O(\|u_n\|)
\\
& \ge c_n \|u_n (1-\psi_n)\|_{L^2(\RD)}^2+O(\|u_n\|),
\end{align*}
where
\[
 c_n=\int_{z_n}^n \frac{u_n^2(r)\psi_n^2(r)}{r}\left(\int_{z_n}^r s u_n^2(s)\psi_n^2(s)\, ds\right) dr\ge c>0.
\]
Therefore, we get
\beq\label{eq:ICR}
I_\o(u_n)\ge I_\o(u_n \psi_n)+I_\o\left(u_n(1-\psi_n)\right)
+c \|u_n (1-\psi_n)\|_{L^2(\RD)}^2+ O(\|u_n\|).
\eeq

\

{\bf Step 4:} The following estimate holds:

\beq \label{comparison} I_\o(u_n \psi_n) = 2\pi \xi_n J_\o(u_n
\psi_n) +O(\|u_n\|). \eeq

\

Indeed, by taking into account Proposition \ref{prop-fund},
\eqref{decay1} and the definition of $\psi_n$ \eqref{def-psi_n},
we have
\begin{align*}
\left|\int_0^n (u_n \psi_n)^2 r\ dr - \xi_n\int_0^n (u_n \psi_n)^2 \ dr \right|
&\le \int_0^n (u_n \psi_n)^2 |r-\xi_n|\ dr
\\
&\le  \int_{\xi_n-3\|u_n\|}^{\xi_n+\|u_n\|} u_n^2 |r-\xi_n|\ dr  +o(1)
\\
&\le  O(\|u_n\|)\int_{\xi_n-3\|u_n\|}^{\xi_n+\|u_n\|} u_n^2\ dr  +o(1)
=O(\|u_n\|).
\end{align*}
The estimates are similar for the other local terms of $I_\o$. For what concerns the nonlocal term, we get
\begin{align*}
&\int_0^n  \frac{(u_n \psi_n)^2(r)}{r}\left(\int_0^r s (u_n \psi_n)^2(s) \, ds\right)^2 dr
-\xi_n\int_0^n  (u_n \psi_n)^2(r)\left(\int_0^r  (u_n \psi_n)^2(s) \, ds\right)^2 dr
\\
&\qquad=\underset{(I)}{\underbrace{\int_0^{n}  (u_n \psi_n)^2(r) \left(\frac{1}{r}-\frac{1}{\xi_n}\right)\left(\int_0^r s (u_n \psi_n)^2(s) \, ds\right)^2 dr}}
\\
&\qquad\quad+\underset{(II)}{\underbrace{\frac 1{\xi_n}\int_0^{n}  (u_n \psi_n)^2(r)\left[\left(\int_0^r  s (u_n \psi_n)^2(s) \, ds\right)^2-\left(\int_0^r  \xi_n (u_n \psi_n)^2(s) \, ds\right)^2 \right] dr}},
\end{align*}
where
\beq\label{(I)}
(I)\le \int_{\xi_n-3\|u_n\|}^{\xi_n+\|u_n\|}  u_n^2(r) \frac{|\xi_n-r|}{r \xi_n}\left(\int_{\xi_n-3\|u_n\|}^{\xi_n+\|u_n\|}  s u_n^2(s) \, ds\right)^2 dr
+o(1)
=O(\|u_n\|)
\eeq
and
\begin{align}
(II)&\le
\frac 1{\xi_n}\int_{\xi_n-3\|u_n\|}^{\xi_n+\|u_n\|} u_n^2(r)\left|\int_{\xi_n-3\|u_n\|}^{\xi_n+\|u_n\|}  (s+\xi_n) u_n ^2(s) \, ds\right|
\left|\int_{\xi_n-3\|u_n\|}^{\xi_n+\|u_n\|}  (s-\xi_n) u_n^2(s) \, ds\right| dr\label{(II)}
\\
&\quad+o(1)\nonumber
\\
&=O(\|u_n\|). \nonumber
\end{align}

\

{\bf Step 5:} Conclusion for $\o > \o_0$.

\

By \eqref{eq:ICR} and \eqref{comparison}, we have
\beq\label{eq:ICR2} I_\o(u_n)\ge 2 \pi \xi_n J_\o(u_n
\psi_n)+I_\o(u_n(1-\psi_n)) +c \|u_n (1-\psi_n)\|_{L^2(\RD)}^2+
O(\|u_n\|). \eeq

Recall that $\|u_n \psi_n\|^2_{H^1(\R)} \ge \s>0$. By Proposition
\ref{extra}, we have that $J_\o(u_n \psi_n) \to c>0$, up to a
subsequence. Since $\xi_n \sim \|u_n\|^2$, it turns out from
\eqref{eq:ICR2} that $I_\o(u_n) > I_\o(u_n(1-\psi_n))$. But this
is a contradiction with the definition of $u_n$, which proves that
$\inf I_\o > -\infty$.

Let us now show that $I_{\o}$ is coercive. Indeed, take $u_n \in
H^1(\RD)$ an unbounded sequence, and assume that $I_{\o}(u_n)$ is
bounded from above. By Proposition \ref{prop-fund}, (iii), we
would obtain that $I_{\hat{\o}}(u_n) \to -\infty$ for any $\o_0<\hat{\o}
< \o$, a contradiction.

\

{\bf Step 6:} Conclusion for $\o = \o_0$.

\

As above, \eqref{eq:ICR2} gives a contradiction unless $J_\o(u_n
\psi_n) \rightarrow 0$. Proposition \ref{extra} now implies that
$\psi_n u_n (\cdot - t_n)\to w_{k_2}$ up to a subsequence, for some
$t_n \in (0,+\infty)$. Since $\xi_n \in D_n$ (see their definition
in \eqref{def-D_n}), we have that $|t_n-\xi_n|$ is bounded. With
this extra information, we have a better estimate of the decay of
the solutions: indeed,

\beq \label{decay2} |u_n(r)| + |u_n'(r)|<C {\rm exp}\left(- \sqrt{\o}
|r-\xi_n|\right),\quad \mbox{ for all } r>\xi_n-2\|u_n\|. \eeq

This allows us to do the cut-off procedure in a much more accurate way. Indeed, take
$$\tilde{z}_n= \xi_n-\|u_n\|. $$  Then, \eqref{decay2} implies that

\beq \label{new} \|u_n\|^2_{H^1(\tilde{z}_n, \tilde{z}_n+1)} \le C
{\rm exp}(- \sqrt{\o} \|u_n\|). \eeq

Define $\tilde{\psi}_n:[0,+\infty] \to [0,1]$ accordingly:
\[% \label{def-psi_n2}
\tilde{\psi}_n(r)=\left\{
\begin{array}{ll}
0, & \hbox{if }r\le \tilde{z}_n,
\\
1, & \hbox{if }  r \ge \tilde{z}_n+1,
\end{array} \ \ |\tilde{\psi}_n'(r)|\le 2.
\right.
\]

The advantage is that, in the estimate of $I_\o(u_n)$, now the
errors are exponentially small. Indeed, by repeating the estimates
of Step 3 with the new information \eqref{new}, we obtain:
\[%\label{eq:ICR3}
 I_\o(u_n)\ge I_\o(u_n
\tilde\psi_n)+I_\o(u_n(1-\tilde\psi_n)) +c \|u_n (1-\tilde\psi_n)\|_{L^2(\RD)}^2+
o(1).
\]

Let us show that in this case \eqref{comparison} becomes
\[
 I_\o(u_n \tilde\psi_n) = 2 \pi \xi_n J_\o(u_n \tilde\psi_n)
 +O(1).
\]
Indeed, by \eqref{decay2} and \eqref{new}, we have
\[
\left|\int_0^n (u_n \tilde\psi_n)^2 r\ dr - \xi_n\int_0^n (u_n \tilde\psi_n)^2 \ dr \right|
\le \ir  (u_n \tilde\psi_n)^2 |r-\xi_n|\ dr \le C;
\]
the other local terms can be estimated similarly. For what concerns the nonlocal term, we repeat the arguments of the previous case using in \eqref{(I)} and \eqref{(II)} the informations contained  in \eqref{decay2} and \eqref{new}.
\\
Then,
\begin{align*}
I_\o(u_n)&\ge I_\o(u_n \tilde{\psi}_n)+I_\o\big(u_n(1-\tilde{\psi}_n)\big)
+c \|u_n (1-\tilde{\psi}_n)\|_{L^2(\RD)}^2+ O(1)
\\
& = 2 \pi\xi_n J_\o(u_n
\tilde{\psi}_n)+I_\o\big(u_n(1-\tilde{\psi}_n)\big) +c \|u_n
(1-\tilde{\psi}_n)\|_{L^2(\RD)}^2+ O(1)
\\
&\ge  I_{(\o+2c)} \big(u_n(1-\tilde{\psi}_n)\big)+ O(1).
\end{align*}
But, by Case 1, we already know that $I_{(\o+2c)}$ is bounded from below, and hence $\inf I_{\o_0} > -\infty$.

Finally,  applying Lemma \ref{le:asin} to $U=w_{k_2}$, we readily get
that $I_{\o_0}$ is not coercive.

\end{proof}

\begin{proof}[Proof of Theorem \ref{teo2}] We shall prove each assessment separately.

{\bf Proof of (i).}
Let $u$ be a solution of \eqref{equation}.
We multiply \eqref{equation} by $u$ and integrate: taking into account the inequality \eqref{ineq}, we get
\begin{align*}
0&= \ird \left(|\n u|^2+\o u^2\right) dx
+\frac 34\ird \frac{u^2(x)}{|x|^2}\left(\int_0^{|x|}s u^2(s)\ ds \right)^2 dx
- \ird |u|^{p+1} dx
\\
&\ge \frac 14\ird |\n u|^2 dx
+\ird \left(\o u^2+\frac 34 u^4-|u|^{p+1}\right)dx.
\end{align*}
Observe that there exists $\bar{\o}>0$ such that, for $\o > \bar{\o}$, the function $ t \mapsto \o t^2+\frac 34 t^4-|t|^{p+1}$ is non-negative. Therefore $u$ must be identically zero.

\
\\
{\bf Proof of (ii).} First, we observe that since $\inf
I_{\o_0}<0$, there exists $\tilde{\o}>\o_0$ such that $\inf I_\o<0$
if and only if $\o\in (\o_0,\tilde{\o})$. Since, by Theorem
\ref{teo1} and Proposition \ref{prop:weak}, $I_\o$ is coercive and
weakly lower semicontinuous, we infer that the infimum is
attained.
\\
Clearly, $0$ is a local minimum for $I_\o$. Then, if $\o \in
(\o_0,\tilde{\o})$, so the functional satisfies the geometrical
assumptions of the Mountain Pass Theorem, see \cite{ar}. Since
$I_\o$ is coercive, (PS) sequences are bounded. By the compact
embedding of $\Hr$ into $L^{p+1}(\R^2)$ and Proposition
\ref{prop:weak}, standard arguments show that $I_\o$ satisfies the
Palais-Smale condition and so we find a second solution which is
at a positive energy level.

\
\\
{\bf Proof of (iii).}
Let now consider $\o\in (0,\o_0)$. Performing the rescaling $u\mapsto u_\o=\sqrt{\o}\ u(\sqrt{\o}\ \cdot)$, we get
\begin{align*}
I_\o(u_\o) & =  \o \left[ \frac 12 \int_{\R^2} \left(|\nabla u|^2
+  u^2\right) dx
+ \frac{1}{8} \int_{\R^2}
\frac{u^2(x)}{|x|^2}\left(\int_0^{|x|} s u^2(s) \, ds\right)^2 dx \right.
\\
 & \left.\qquad  -
\frac{\o^\frac{p-3}{2}}{p+1}  \int_{\R^2} |u|^{p+1} \, dx\right].
\end{align*}
Define $\l=\o^\frac{p-3}{2}$ and $\mathcal{I}_\l:\Hr \to \R$ as
\begin{align*}
{\mathcal I}_\l(u) & =  \frac 12 \int_{\R^2} \left(|\nabla u|^2
+  u^2\right) dx
+ \frac{1}{8} \int_{\R^2}
\frac{u^2(x)}{|x|^2}\left(\int_0^{|x|} s u^2(s) \, ds\right)^2 dx
\\
 & \quad  -
\frac{\l}{p+1}  \int_{\R^2} |u|^{p+1} \, dx.
\end{align*}
Since $\I_\l$ satisfies the geometrical assmptions of the Mountain Pass Theorem, by \cite[Theorem 1.1]{jeanjean}, we
infer that, for almost every $\l$, the functional $\I_\l$ possesses a bounded Palais-Smale sequence $u_n$.
Assume $u_n \weakto u$; Proposition \ref{prop:weak} and standard arguments imply that $u$ is a critical point of
${\mathcal I}_\l$. Making the change of variables back we obtain a solution of \eqref{equation} for almost every $\o\in (0,\o_0)$.

\
\\
Finally, in order to find positive solutions of \eqref{equation}, we simply observe that the whole argument applies to the functional $I_\o^+:\Hr \to \R$
\begin{align*}
I_\o^+(u) & =  \frac 12 \int_{\R^2} \left(|\nabla u|^2
+ \o u^2\right) dx
+ \frac{1}{8} \int_{\R^2}
\frac{u^2(x)}{|x|^2}\left(\int_0^{|x|} s u^2(s) \, ds\right)^2 dx
\\
 & \quad  -
\frac{1}{p+1}  \int_{\R^2} (u^+)^{p+1} \, dx.
\end{align*}
Due to the maximum principle, the critical points of $I_\o^+$ are positive solutions of \eqref{equation}.

\end{proof}

\end{document}